\documentclass[11pt, notitlepage]{article}   

\usepackage{amsmath,amsthm,amsfonts}   

\usepackage{amscd}
\usepackage{amsfonts}
\usepackage{amssymb}
\usepackage{color}      
\usepackage{epsfig}
\usepackage{graphicx}           

\usepackage{epsfig}
\include{psfig}

\theoremstyle{plain}
\newtheorem{theorem}{Theorem}[section]
\newtheorem{lemma}[theorem]{Lemma}
\newtheorem{corollary}[theorem]{Corollary}
\newtheorem{proposition}[theorem]{Proposition}
\newtheorem{Obs}[theorem]{Observation}
\newtheorem{remark}[theorem]{Remark}

\theoremstyle{definition}

\def\finf{\mathop{{\rm I}\kern -.27 em {\rm F}}\nolimits}

\begin{document}

\title{On Zero Forcing Number of Graphs and Their Complements}

\author{{\bf Linda Eroh$^1$}, {\bf Cong X. Kang$^2$}, and {\bf Eunjeong Yi$^3$}\\
\small$^1$ University of Wisconsin Oshkosh, Oshkosh, WI 54901, USA\\
\small $^{2,3}$Texas A\&M University at Galveston, Galveston, TX 77553, USA\\
$^1${\small\em eroh@uwosh.edu}; $^2${\small\em kangc@tamug.edu}; $^3${\small\em yie@tamug.edu}}

\maketitle

\date{}

\begin{abstract}
The \emph{zero forcing number}, $Z(G)$, of a graph $G$ is the minimum cardinality of a set $S$ of black vertices (whereas vertices in $V(G) \setminus S$ are colored white) such that $V(G)$ is turned black after finitely many applications of ``the color-change rule": a white vertex is converted to a black vertex if it is the only white neighbor of a black vertex. Zero forcing number was introduced and used to bound the minimum rank of graphs by the ``AIM Minimum Rank -- Special Graphs Work Group". It's known that $Z(G)\geq \delta(G)$, where $\delta(G)$ is the minimum degree of $G$. We show that $Z(G)\leq n-3$ if a connected graph $G$ of order $n$ has a connected complement graph $\overline{G}$. Further, we characterize a tree or a unicyclic graph $G$ which satisfies either $Z(G)+Z(\overline{G})=\delta(G)+\delta(\overline{G})$ or $Z(G)+Z(\overline{G})=2(n-3)$.
\end{abstract}


\vskip .1in


\section{Introduction}

Let $G = (V(G),E(G))$ be a finite, simple, undirected, connected graph of order $|V(G)|=n \ge 2$ and size $|E(G)|$. For $W \subseteq V(G)$, we denote by $G[W]$ the subgraph of $G$ induced by $W$. For a vertex $v \in V(G)$, the \emph{open neighborhood of $v$} is the set $N_G(v)=\{u \mid uv \in E(G)\}$. The \emph{degree} $\deg_G(v)$ of a vertex $v \in V(G)$ is the the number of edges incident to the vertex $v$ in $G$; a \emph{leaf} is a vertex of degree one. We denote by $\Delta(G)$ the \emph{maximum degree} of a graph $G$, and denote by $\delta(G)$ the \emph{minimum degree} of a graph $G$. We denote by $K_n$, $C_n$, and $P_n$ the complete graph, the cycle, and the path, respectively, on $n$ vertices. The \emph{distance} between two vertices $v, w \in V(G)$, denoted by $d_G(v, w)$, is the length of a shortest path between $v$ and $w$; we omit $G$ when ambiguity is not a concern. The \emph{diameter}, $diam(G)$, of a graph $G$ is given by $\max\{d(u, v) \mid u,v \in V(G)\}$. The \emph{complement} $\overline{G}$ of a graph $G$ is the graph whose vertex set is $V(G)$ and $uv \in E(\overline{G})$ if and only if $uv \not\in E(G)$ for $u,v \in V(G)$. For other terms in graph theory, refer to \cite{CZ}.

The notion of a zero forcing set, as well  as the associated zero forcing number, of a simple graph was introduced by the ``AIM Minimum Rank -- Special Graphs Work Group" in~\cite{AIM} to bound the minimum rank of associated matrices for numerous families of graphs. Let each vertex of a graph $G$ be given one of two colors, ``black" and ``white" by convention. Let $S$ denote the (initial) set of black vertices of $G$. The \emph{color-change rule} converts the color of a vertex from white to black if the white vertex $u_2$ is the only white neighbor of a black vertex $u_1$; we say that $u_1$ forces $u_2$, which we denote by $u_1 \rightarrow u_2$. And a sequence, $u_1 \rightarrow u_2 \rightarrow \cdots \rightarrow u_{i} \rightarrow u_{i+1} \rightarrow \cdots \rightarrow u_t$, obtained through iterative applications of the color-change rule is called a \emph{forcing chain}. Note that, at each step of the color change, there may be two or more vertices capable of forcing the same vertex. The set $S$ is said to be \emph{a zero forcing set} of $G$ if all vertices of $G$ will be turned black after finitely many applications of the color-change rule. The \emph{zero forcing number} of $G$, denoted by $Z(G)$, is the minimum of $|S|$ over all zero forcing sets $S \subseteq V(G)$.

Since its introduction by the aforementioned ``AIM group", zero forcing number has become a graph parameter studied for its own sake, as an interesting invariant of a graph. The four authors in~\cite{iteration} studied the number of steps it takes for a zero forcing set to turn the entire graph black; they named this new graph parameter the \emph{iteration index} of a graph: from a ``real world" modeling (or discrete dynamical system) perspective, if the initial black set is capable of passing a certain condition or trait to the entire population (i.e. ``zero forcing"), then the iteration index of a graph may represent the number of units of time (anything from days to millennia) necessary for the entire population to acquire the condition or trait. Independently, Hogben et al. studied the same parameter (iteration index) in~\cite{proptime}, which they called \emph{propagation time}. It's also noteworthy that physicists have independently studied the zero forcing parameter, referring to it as the \emph{graph infection number}, in conjunction with the control of quantum systems (see \cite{p1}, \cite{p2}, and \cite{p3}). More recently, a probabilistic interpretation of zero forcing was introduced in \cite{pzf}, and a comparative study of metric dimension and zero forcing number for graphs was initiated in~\cite{dimZ}. For more articles and surveys pertaining to the zero forcing parameter, see \cite{pathcover, min-degree, iteration, Z+e, ZFsurvey, ZFsurvey2, cutvertex}.

In this paper, we obtain a Nordhaus-Gaddum-type result (see \cite{nordhaus}) on zero forcing number of graphs by first showing that $Z(G)\leq n-3$ if both $G$ and $\overline{G}$ are connected graphs of order $n$. It's known that $Z(G)\geq \delta(G)$; thus, $\delta(G)+\delta(\overline{G}) \le Z(G)+Z(\overline{G}) \le 2(n-3)$ for connected graphs $G$ and $\overline{G}$ of order $n$. Further, we characterize a tree or a unicyclic graph $G$ which satisfies either $Z(G)+Z(\overline{G})=\delta(G)+\delta(\overline{G})$ or $Z(G)+Z(\overline{G})=2(n-3)$.


\section{Bounds for $Z(G)+Z(\overline{G})$}

The \emph{path cover number} $P(G)$ of $G$ is the minimum number of vertex disjoint paths, occurring as induced subgraphs of $G$, that cover all the vertices of $G$. First, we recall some results on zero forcing number of graphs.

\begin{theorem} \label{pathcover}
\begin{itemize}
\item[(a)] \cite{pathcover} For any graph $G$, $P(G) \le Z(G)$.
\item[(b)] \cite{AIM} For any tree $T$, $P(T) = Z(T)$.
\item[(c)] \cite{cutvertex} For any unicyclic graph $G$, $P(G)=Z(G)$.
\end{itemize}
\end{theorem}

\begin{theorem}\cite{min-degree}\label{mindegree}
For any graph $G$ of order $n \ge 2$, $Z(G) \ge \delta(G)$.
\end{theorem}

\begin{theorem} \cite{dimZ, cutvertex} \label{ObsZ} Let $G$ be a connected graph of order $n \ge 2$. Then
\begin{itemize}
\item[(a)] $Z(G)=1$ if and only if $G=P_n$;
\item[(b)] $Z(G)=n-1$ if and only if $G=K_n$.
\end{itemize}
\end{theorem}

\begin{theorem}\cite{cutvertex} \label{cutV}
Let $G$ be a graph with cut-vertex $v \in V(G)$. Let $V_1, V_2, \ldots, V_k$ be the vertex sets for the connected components of $G[V(G) \setminus \{v\}]$, and for $1\le i \le k$, let $G_i$ = $G[V_i \cup \{v\}]$. Then $Z(G) \ge  1-k+\sum_{i=1}^{k} Z(G_i)$.
\end{theorem}

\begin{theorem}\label{connected}
Let $G$ and $\overline{G}$ be connected graphs of order $n \ge 4$. Then $Z(G) \le n-3$.
\end{theorem}

\begin{proof}
Let $G$ and $\overline{G}$ be connected graphs of order $n \ge 4$. Since $\overline{G}$ is connected, $\Delta(G) \le n-2$. If $\Delta(G)=1$, then $G \cong P_2$, and thus $\Delta(G) \ge 2$. We consider two cases.

\emph{Case 1: $\Delta(G)=n-2$.} Let $V(G)=\{u_1, u_2\} \cup W$, where $W=\{w_i \mid 1 \le i \le n-2\}$. Suppose that $\deg_G(u_1)=\Delta(G)=n-2$ and let $N_G(u_1)=W$. If $u_2w_j \in E(G)$ for each $j$ ($1 \le j \le n-2$), then $G$ contains the complete bi-partite graph $K_{2, n-2}$ as a subgraph, and thus $\overline{G}$ is disconnected. Next, suppose there exists $w_j$ such that $u_2w_j \not\in E(G)$ for some $j$ ($1 \le j \le n-2$). Without loss of generality, we may assume that $N_G(u_2) \cap W=\{w_i \mid 1 \le i \le k\}$ for some $k<n-2$. If $w_iw_j \in E(G)$ for each $i, j$ ($1 \le i \le k$ and $k+1 \le j \le n-2$), then $G$ contains the complete bi-partite graph $K_{k, n-k}$ as a subgraph, and thus $\overline{G}$ is disconnected. So, there exists two vertices $w_x$ and $w_y$ such that $w_xw_y \not\in E(G)$, where $1 \le x \le k$ and $k+1 \le y \le n-2$. Without loss of generality, we may assume that $w_kw_{n-2} \not\in E(G)$, by relabeling if necessary (see (a) of Fig.~\ref{conn}). Then $V(G) \setminus \{u_1, w_k, w_{n-2}\}$ forms a zero forcing set for $G$: $u_2 \rightarrow w_k \rightarrow u_1 \rightarrow w_{n-2}$. Thus, $Z(G) \le n-3$.

\emph{Case 2: $\Delta(G) =n-a$, where $3 \le a \le n-2$.} Let $V(G)=U \cup W$, where $U=\{u_i \mid 1 \le i \le a\}$ and $W=\{w_j \mid 1 \le j \le n-a\}$ for $3 \le a \le n-2$. Let $N_G(u_1)=W$; so that $\deg_G(u_1)=\Delta(G)=n-a$. For $\overline{G}$ to be connected, $G$ can not contain $K_{a, n-a}$ as a subgraph, meaning $u_iw_j \not\in E(G)$ for a pair $(i,j)$ with $1<i \le a$ and $1 \le j \le n-a$. First, suppose there is a $w_k$ such that $u_{\alpha}w_k \in E(G)$ and $u_{\beta}w_k \not\in E(G)$, where $2 \le \alpha, \beta \le a$ (see (b) of Fig.~\ref{conn}). Then $V(G) \setminus \{u_{\alpha}, u_{\beta}, w_k\}$ forms a zero forcing set for $G$, since $u_1 \rightarrow w_{k} \rightarrow u_{\alpha}$ and $v \rightarrow u_{\beta}$ for some vertex $v \in V(G)$ with $vu_{\beta} \in E(G)$; here, we note that such a vertex $v$ exists by the connectedness of $G$. Next, suppose $w_k$ as above does not exist. Then $W=W' \cup W''$, where $W'=\{w \in W \mid N_G(w) \cap U=\{u_1\}\}$ and $W''=\{w \in W \mid N_G(w) \cap U=U\}$. As already noted, $W' \neq \emptyset$. Notice also that $W' \neq W$ (i.e., $W'' \neq \emptyset$), since $N_G(u_1)=W$ and $G$ is connected. If there exist $w_{\alpha}\in W'$ and $w_{\beta} \in W''$ such that $w_{\alpha}w_{\beta} \not\in E(G)$, then $V(G) \setminus \{u_1, u_a, w_{\beta}\}$ is a zero forcing set, since $w_{\alpha} \rightarrow u_1 \rightarrow w_{\beta} \rightarrow u_a$ (see (c) of Fig.~\ref{conn}). If, for all $(w_x, w_y) \in W' \times W''$, $w_xw_y \in E(G)$, then $G$ contains the complete bi-partite graph $K_{|W''|, |W'|+a}$ with bi-partite sets $W''$ and $U \cup W'$, and $\overline{G}$ will not be connected. Thus, in each case, $Z(G) \le n-3$ if both $G$ and $\overline{G}$ are connected.~\hfill
\end{proof}

\begin{figure}[htbp]
\begin{center}
\begin{picture}(0,0)(80,55)
\setlength{\unitlength}{1.3pt}

\put(-60,-20){(a) $w_kw_{n-2} \not\in E(G)$}
\put(25,-20){(b) $u_{\beta}w_k \not\in E(G)$}
\put(108,-20){(c) $w_{\alpha}w_{\beta} \not\in E(G)$}
\put(-37,-1){$\cdots$}
\put(-22,-1){$\cdots$}
\put(42,-1){$\cdots$}
\put(58,-1){$\cdots$}
\put(68,-1){$\cdot$}
\put(122,-1){$\cdots$}
\put(136,-1){$\cdots$}
\put(151,-1){$\cdots$}
\put(161,-1){$\cdot$}
\put(33,29){$\ldots$}
\put(48,29){$\ldots$}
\put(63,29){$\ldots$}
\put(133,29){$\ldots$}

\put(-49,0){\circle{3}}
\put(-53,-6){$w_1$}
\put(-40,0){\circle{3}}
\put(-43,-6){$w_2$}
\put(-25,0){\circle{3}}
\put(-28,-6){$w_k$}
\put(-11,0){\circle{3}}
\put(-14,-6){$w_{n-2}$}
\put(-35,30){\circle{3}}
\put(-38,34){$u_1$}
\put(-20,30){\circle{3}}
\put(-23,34){$u_2$}

\put(29,0){\circle{3}}
\put(25,-6){$w_1$}
\put(39,0){\circle{3}}
\put(35,-6){$w_2$}
\put(55,0){\circle{3}}
\put(51,-6){$w_k$}
\put(75,0){\circle{3}}
\put(71,-6){$w_{n-a}$}
\put(20,30){\circle{3}}
\put(16,34){$u_1$}
\put(30,30){\circle{3}}
\put(26,34){$u_2$}
\put(45,30){\circle{3}}
\put(41,34){$u_{\alpha}$}
\put(60,30){\circle{3}}
\put(56,34){$u_{\beta}$}
\put(75,30){\circle{3}}
\put(71,34){$u_a$}

\put(110,0){\circle{3}}
\put(106,-6){$w_1$}
\put(120,0){\circle{3}}
\put(116,-6){$w_2$}
\put(134,0){\circle{3}}
\put(130,-6){$w_{\alpha}$}
\put(149,0){\circle{3}}
\put(145,-6){$w_{\beta}$}
\put(166,0){\circle{3}}
\put(162,-6){$w_{n-a}$}
\put(110,30){\circle{3}}
\put(106,34){$u_1$}
\put(120,30){\circle{3}}
\put(116,34){$u_2$}
\put(130,30){\circle{3}}
\put(126,34){$u_3$}
\put(145,30){\circle{3}}
\put(141,34){$u_a$}

\put(-35,28.5){\line(-1,-2){13.5}}
\put(-35,28.5){\line(-1,-5){5.5}}
\put(-35,28.5){\line(1,-3){9}}
\put(-35,28.5){\line(5,-6){23}}
\put(-20,28.5){\line(-1,-1){27.5}}
\put(-20,28.5){\line(-3,-4){20}}
\put(-20,28.5){\line(-1,-5){5.5}}

\put(20,28.5){\line(1,-3){9}}
\put(20,28.5){\line(2,-3){18}}
\put(20,28.5){\line(5,-4){34}}
\put(20,28.5){\line(5,-4){23}}
\put(20,28.5){\line(2,-1){54}}
\put(55,1.5){\line(-1,3){9}}

\put(110,28.5){\line(0,-1){27}}
\put(110,28.5){\line(1,-3){9}}
\put(110,28.5){\line(5,-6){23}}
\put(110,28.5){\line(4,-3){37.5}}
\put(110,28.5){\line(2,-1){55}}
\put(120,28.5){\line(1,-1){28}}
\put(120,28.5){\line(5,-3){45}}
\put(130,28.5){\line(2,-3){18}}
\put(130,28.5){\line(4,-3){36}}
\put(145,28.5){\line(1,-6){4.5}}
\put(145,28.5){\line(4,-5){22}}

\end{picture}
\end{center}
\vspace{1in}
\caption{Connected graphs $G$ of order $n \ge 4$ with $2 \le \Delta(G) \le n-2$}\label{conn}
\end{figure}
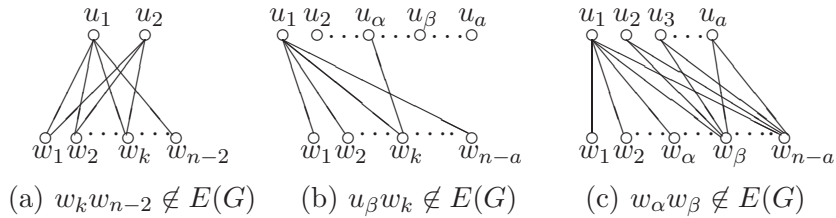

\begin{remark}
The bound obtained in Theorem \ref{connected} cannot be improved. For example, $G= \overline{C_n}$ and $\overline{G}=C_n$ are connected for $n \ge 5$ and $Z(G)\ge \delta(G)=n-3$.
\end{remark}

\begin{remark}
We note that Theorem~\ref{connected} can be also deduced as follows: By Theorem~5.4 of~\cite{AIM}, $Z(G)\geq n-2$ implies $G$ does not contain $P_4$ as an induced subgraph. Thence, $G$ is a \emph{cograph} (i.e., \emph{complement reducible graph}) and an equivalent characterization of a cograph $G$ is this: the complement of any nontrivial connected induced subgraph of $G$ is disconnected (see Theorem 2 of~\cite{CRG} for details). However, the proof of Theorem~\ref{connected} first provided bears the virtue of being simple, direct, and completely self-contained.
\end{remark}

Theorems \ref{mindegree} and \ref{connected} imply a Nordhaus-Gaddum-type result on zero forcing number of graphs as follows.

\begin{corollary}\label{bounds}
Let $G$ and $\overline{G}$ be connected graphs of order $n \ge 4$. Then
$$\delta(G)+n-1-\Delta(G)=\delta(G)+\delta(\overline{G})\le Z(G)+ Z(\overline{G}) \le 2(n-3),$$ and both bounds are sharp.
\end{corollary}

\begin{proof}
Let $G$ and $\overline{G}$ be connected graphs of order $n \ge 4$. Then $1 \le Z(G), Z(\overline{G}) \le n-3$ by Theorem~\ref{connected}, and thus the upper bound follows. The lower bound comes from Theorem~\ref{mindegree} which yields $Z(G)+Z(\overline{G}) \ge \delta(G)+\delta(\overline{G})$, together with the observation that $\delta(\overline{G})+\Delta(G)=n-1$. For the sharpness of the lower bound, refer to section 3. For the sharpness of the upper bound, refer to sections 4 and 5.~\hfill
\end{proof}

\begin{remark} \label{remark} Let $G$ and $\overline{G}$ be connected graphs of order $n \geq 4$. Then
\begin{itemize}
\item[(a)] $Z(G)+Z(\overline{G})=\delta(G)+\delta(\overline{G})$ is equivalent to $Z(G)=\delta(G)$ and $Z(\overline{G})=\delta(\overline{G})$;
\item[(b)] $Z(G)+Z(\overline{G})=2(n-3)$ is equivalent to $Z(G)=n-3=Z(\overline{G})$.
\end{itemize}
\end{remark}

In the rest of this paper, we characterize when $Z(G)+Z(\overline{G})$ achieves the lower bound or the upper bound of Corollary~\ref{bounds} in the case where $G$ is a tree or a unicyclic graph.


\section{Characterization of $Z(G)+Z(\overline{G})=\delta(G)+\delta(\overline{G})$ when $G$ is a tree or a unicyclic graph}

\begin{proposition}
Let $G$ be a graph of order $n$ with $\delta(G)=1$. Then $Z(G)+Z(\overline{G})=\delta(G)+\delta(\overline{G})$ if and only if $G=P_n$, the path on $n \ge 4$ vertices.
\end{proposition}

\begin{proof}
(Obvious.)
\end{proof}

A graph is \emph{unicyclic} if it contains exactly one cycle. Note that a connected graph $G$ is unicyclic if and only if $|E(G)|=|V(G)|$. Next, we consider the case when $G$ is a unicyclic graph.

\begin{theorem}
Let $G$ be a unicyclic graph of order $n$. Then $Z(G)+Z(\overline{G})=\delta(G)+\delta(\overline{G})$ if and only if $G=C_n$, the cycle on $n \ge 5$ vertices.
\end{theorem}

\begin{proof}
($\Longrightarrow$) Since $G$ is unicyclic (i.e., $G \neq P_n$), we have $2 \le Z(G) = \delta (G) \le 2$ by Theorem~\ref{ObsZ}(a) and Remark~\ref{remark}(a). Since $\delta(G)=2$, $G$ must be $C_n$, where $n \ge 5$ since $Z(\overline{G})=\delta(\overline{G})$ (implying the connectedness of $\overline{G}$).

($\Longleftarrow$) If $G=C_n$, $n \ge 5$, then since any two adjacent vertices of a cycle form a minimum zero forcing set, $Z(G)=2=\delta(G)$. By Theorems \ref{mindegree} and \ref{connected},  $\delta(\overline{C_n}) \le Z(\overline{C_n}) \le n-3=\delta(\overline{C_n})$. ~\hfill
\end{proof}


\section{Characterization of $Z(G)+Z(\overline{G})=2(n-3)$ when $G$ is a tree}

In this section, we characterize trees $T$ and their complements $\overline{T}$ such that $Z(T)+Z(\overline{T})$ achieves the upper bound of Corollary \ref{bounds}. We first recall the following definitions, which can be found in~\cite{CEJO}.

Fix a graph $G$. A vertex of degree at least three is called a \emph{major vertex}. A leaf $u$ is called \emph{a terminal vertex of a major vertex} $v$ if $d(u, v)<d(u, w)$ for every other major vertex $w$. The \emph{terminal degree} of a major vertex $v$ is the number of terminal vertices of $v$. A major vertex $v$ is an \emph{exterior major vertex} if it has positive terminal degree.

\begin{Obs}(c.f. Prop. 4.4 of~\cite{AIM})\label{obs(-1)}
The presence of long vertex-disjoint path(s) indicates, by the fact that $Z(T)=P(T)$ (Theorem~\ref{pathcover}(b)), an upper bound for $Z(T)$ in terms of the order of $T$. For example, if a tree $T$ of order $n$ contains two vertex-disjoint paths $P^1$ and $P^2$ of lengths $4$ and $3$, then $Z(T)=P(T)\leq n-7$ since there is a path cover for $T$ consisting of $P^1$, $P^2$, and the other $n-9$ vertices, each as a path of length $0$.
\end{Obs}

\begin{Obs}\label{observation}
Let $G$ be a graph of order $n \ge 4$.
\begin{itemize}
\item[(a)] If $G$ is $P_n$, then $Z(G)+Z(\overline{G}) = 2(n-3)$ if and only if $n=4$.
\item[(b)] If $G$ is $C_n$, then $Z(G)+Z(\overline{G}) = 2(n-3)$ if and only if $n=5$.
\end{itemize}
\end{Obs}

\begin{lemma}\label{GK}
Let a graph $G$ contain as a subgraph the complete graph $K_m$ on $m\geq 2$ vertices. Then $Z(G)\geq Z(K_m)=m-1$.
\end{lemma}

\begin{proof}
Let $H$ be a fixed $K_m$ in $G$. Let $u_{i,1} \rightarrow u_{i,2} \rightarrow \ldots \rightarrow u_{i,s(i)}$, where $1\leq i\leq m-2$, be $m-2$ forcing chains where $u_{i,1}$ is the initial black vertex of the $i$-th chain and $u_{i,s(i)}$ is the first vertex of the $i$-th chain in $H$. Since there are at most $m-2$ black vertices in $H$, none of the two or more white vertices of $H$ will be forced black: each of the ($m-2$) or fewer black vertices of $H$ has at least 2 white neighbors in $H$.~\hfill
\end{proof}

\begin{remark} For a graph $G$, let $M(G)$ be the maximum nullity of the associated matrices of $G$, $\omega(G)$ the clique number of $G$, and $h(G)$ the Hadwiger number of $G$.
\begin{itemize}
\item[(a)] It is shown that $Z(G) \ge M(G)$ (\cite{AIM}) and that $M(G) \ge \omega(G)-1$ (\cite{ref1}), and thus implying Lemma \ref{GK}.
\item[(b)] It is shown in \cite{ref2} that $M(G) \ge h(G)-1$, which implies Lemma \ref{GK} since $h(G) \ge \omega(G)$.
\end{itemize}
\end{remark}

\begin{theorem} \label{T+}
Let $G$ be a tree of order $n\geq 5$. If $Z(G)=n-3$, then $G$ is the graph obtained by subdividing one edge of the star $S_{n-1}=K_{1, n-2}$.
\end{theorem}

\begin{proof}
Let $G$ be a tree of order $n\geq 5$. Assume $Z(G)=n-3$. If $G=P_n$, $n \ge 5$, then $Z(G)< n-3$ by Theorem~\ref{ObsZ}(a). If $G$ contains at least two major vertices, $G$ must contain at least two exterior major vertices, say $v_1$ and $v_2$, each with terminal degree at least two. Let $N_G(v_1) \supset \{x_1, x_2\}$ and $N_G(v_2) \supset \{x_3, x_4\}$ such that each $x_i$, $i \in \{1,2,3,4\}$, is not on the $v_1-v_2$ geodesic. Then $x_1,v_1,x_2$ and $x_3, v_2, x_4$ are vertex-disjoint paths in $G$; thus $Z(G) =P(G) \le n-4$. So, $G$ must have exactly one major vertex, say $v$, and $n-3=Z(G)=P(G)=\deg_G(v)-1$ implies $\deg_G(v)=n-2$. Thus $G$ is the graph obtained by subdividing one edge of the star $S_{n-1}$.~\hfill
\end{proof}

\begin{corollary}\label{T+CT}
Let $G$ be a tree of order $n\geq 5$ with a connected $\overline{G}$. Then $Z(G)+Z(\overline{G})=2(n-3)$ if and only if $G$ is the graph obtained by subdividing one edge of the star $S_{n-1}$.
\end{corollary}

\begin{proof}
($\Longrightarrow$) It follows from Theorem \ref{T+}.

($\Longleftarrow$) Let $V(G)=\{v, s, \ell_1, \ell_2, \ldots, \ell_{n-2}\}$ such that $\deg_G(v)=n-2 \ge 3$, $\deg_G(s)=2$, and $\deg_G(\ell_i)=1$, $i \in \{1, 2, \ldots, n-2\}$, with $s\ell_1 \in E(G)$ (see Fig.~\ref{TC}). Then, by Theorem~\ref{pathcover}(b), $Z(G)=P(G)=n-3$; $S=\{\ell_i \mid 2 \le i \le n-2\}$ is a zero forcing set for $G$, since $\ell_{n-2} \rightarrow v \rightarrow s \rightarrow \ell_1$. Next, we note that $Z(\overline{G})=n-3$: (i) $Z(\overline{G}) \ge n-3$ by Lemma~\ref{GK}, since $\overline{G}[\{\ell_i \mid 1 \le i \le n-2\}] \cong K_{n-2}$; (ii) $Z(\overline{G}) \le n-3$ by Theorem \ref{connected}.~\hfill
\end{proof}

\vskip .15in

\begin{figure}[htbp]
\begin{center}
\begin{picture}(0,0)(60,55)
\setlength{\unitlength}{1.3pt}

\put(-1,54){$T$}
\put(79,54){$\overline{T}$}

\put(0,40.5){\circle{3}}
\put(-1,43){$v$}
\put(0,0){\circle*{3}}
\put(-1,-8){${\ell}_3$}
\put(-20,20){\circle{3}}
\put(-26,19){$s$}
\put(20,20){\circle*{3}}
\put(23,19){${\ell}_5$}
\put(-13,34){\circle{3}}
\put(-22,34){${\ell}_1$}
\put(-13,7){\circle*{3}}
\put(-22,5){${\ell}_2$}
\put(13,7){\circle*{3}}
\put(15,4){${\ell}_4$}

\put(80,40.5){\circle*{3}}
\put(79,43){$v$}
\put(80,0){\circle*{3}}
\put(79,-8){${\ell}_3$}
\put(60,20){\circle{3}}
\put(54,19){$s$}
\put(100,20){\circle{3}}
\put(103,19){${\ell}_5$}
\put(67,33){\circle{3}}
\put(58,33){${\ell}_1$}
\put(67,7){\circle*{3}}
\put(58,5){${\ell}_2$}
\put(93,7){\circle*{3}}
\put(95,4){${\ell}_4$}

\put(0,39){\line(0,-1){40}}
\put(0,39){\line(1,-1){18.5}}
\put(0,39){\line(-1,-1){18.5}}
\put(0,39){\line(-2,-5){13}}
\put(0,39){\line(2,-5){13}}
\put(-19.5,21.5){\line(1,2){5.4}}

\put(68.5, 33){\line(3,2){12}}
\put(67.5,31.5){\line(0,-1){25}}
\put(67.5,31.5){\line(1,-1){26}}
\put(67.5,31.5){\line(2,-5){12}}
\put(67.5,31.5){\line(3,-1){31}}
\put(61.5,20){\line(1,0){37}}
\put(61.5,20){\line(1,-2){6}}
\put(61.5,20){\line(5,-2){32}}
\put(61.5,20){\line(1,-1){18}}
\put(67,7){\line(5,2){31.5}}
\put(67,7){\line(1,0){25}}
\put(67,7){\line(2,-1){12}}
\put(80,0){\line(1,1){19}}
\put(80,0){\line(2,1){13}}
\put(93,7){\line(1,2){6}}

\end{picture}
\end{center}
\vspace{.9in}
\caption{The tree $T$ of order $n=7$ satisfying $Z(T)+Z(\overline{T})=2(n-3)$}\label{TC}
\end{figure}
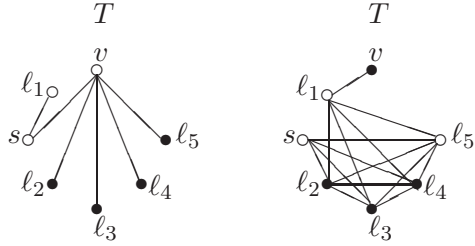


\section{Characterization of $Z(G)+Z(\overline{G})=2(n-3)$ when $G$ is a unicyclic graph}

In this section, we characterize a unicyclic graph $G$ having a connected $\overline{G}$ such that $Z(G)+Z(\overline{G})$ achieves the upper bound of Corollary~\ref{bounds}.

\begin{lemma}\label{5}
Let $G$ and $\overline{G}$ be connected graphs of order 5. If $G$ is a unicyclic graph, then $Z(G)=2=Z(\overline{G})$.
\end{lemma}

\begin{proof}
Since $G$ is not a path and $\overline{G}$ needs to be connected, by Theorems~\ref{ObsZ}(a) and \ref{connected}, $Z(G)=2$. Since $\overline{G}$, the complement of a unicyclic graph in $K_5$, can not be a path and its complement (namely $G$) is connected, again by Theorem~\ref{connected}, $Z(\overline{G})=2$.~\hfill
\end{proof}

\begin{theorem}\label{main}
Let $G$ be a connected, unicyclic graph of order $n\geq 6$ and having a connected $\overline{G}$. Then $Z(G)=n-3$ if and only if $G$ is the vertex sum of $C_3$ and $S_{n-2}$ at one of the leaves of the star.
\end{theorem}

\begin{proof}
Let $G$ be a connected, unicyclic graph of order $n\geq 6$. Assume $Z(G)=n-3$. We first make the following\\

\textbf{Claim:} $diam(G)=3$.

\textit{Proof of Claim.} By Theorem~\ref{pathcover}(c), $diam(G) \le 3$. If $diam(G)=1$, then $G \cong K_n$ with $Z(G)=n-1$. If $G$ is unicyclic and $diam(G)=2$, then $G \in \{C_5, C_4, H\}$, where $H$ is the vertex sum of $C_3$ and $S_{n-2}$ at the major vertex of the star. Since $n \ge 6$ and $\overline{H}$ is disconnected by the fact that $\Delta(H)=n-1$, $diam(G) \ge 3$. Thus $diam(G)=3$.~$\Box$\\

Let $\mathcal{C}=C_m$ be the unique cycle of $G$. Note that $d_{\mathcal{C}}(x,y)=d_G(x,y)$ for $x, y \in \mathcal{C}$. Since $diam(G)=3$, $diam(\mathcal{C}) \le 3$, and hence $3 \le m \le 7$. If $m \in \{6,7\}$, then $G \in \{C_6, C_7\}$ and $Z(G)=2<n-3$. If $m=5$, then $G$ is isomorphic to (a) or (b) of Fig.~\ref{uniC}; in each case, $\ell, v_1, v_2, v_3, v_4$ is an induced path in $G$, and thus, by Theorem~\ref{pathcover}(c), $Z(G)=P(G) \le n-4$. If $m=4$, then $G$ is isomorphic to (c) or (d) of Fig.~\ref{uniC}. If $G$ is isomorphic to (c) of Fig.~\ref{uniC}, then $\ell_1, v_1, \ell_2$ (notice that $t \ge 2$ since $n \ge 6$) and $v_2, v_3, v_4$ are induced paths in $G$; thus $Z(G)=P(G) \le n-4$. If $G$ is isomorphic to (d) of Fig.~\ref{uniC}, then $\ell_1, v_1, v_2, \ell_2$ and $v_3, v_4$ are induced paths in $G$; thus $Z(G)=P(G) \le n-4$.

\begin{figure}[htbp]
\begin{center}
\begin{picture}(0,0)(60,55)
\setlength{\unitlength}{1.3pt}

\put(-55,-12){(a)}
\put(5,-12){(b)}
\put(65,-12){(c)}
\put(125,-12){(d)}
\put(-49,35){$\ldots$}
\put(11,35){$\ldots$}
\put(72,34){$\ldots$}
\put(132,34){$\ldots$}
\put(30,10){$\cdot$}
\put(30,7){$\cdot$}
\put(30,4){$\cdot$}
\put(150,7){$\cdot$}
\put(150,4){$\cdot$}
\put(150,1){$\cdot$}

\put(-50,25){\circle{3}}
\put(-47.5,24){$v_1$}
\put(-60,15){\circle{3}}
\put(-69,15){$v_5$}
\put(-40,15){\circle{3}}
\put(-38,14){$v_2$}
\put(-60,0){\circle{3}}
\put(-69,-1){$v_4$}
\put(-40,0){\circle{3}}
\put(-38,-1){$v_3$}
\put(-50,36){\circle{3}}
\put(-58,36){\circle{3}}
\put(-64,35){$\ell$}
\put(-38,36){\circle{3}}

\put(10,25){\circle{3}}
\put(12.5,24){$v_1$}
\put(0,15){\circle{3}}
\put(-9,15){$v_5$}
\put(20,15){\circle{3}}
\put(11,14){$v_2$}
\put(0,0){\circle{3}}
\put(-9,-1){$v_4$}
\put(20,0){\circle{3}}
\put(17,-6){$v_3$}
\put(10,36){\circle{3}}
\put(2,36){\circle{3}}
\put(-4,35){$\ell$}
\put(22,36){\circle{3}}
\put(31,15){\circle{3}}
\put(31,23){\circle{3}}
\put(31,3){\circle{3}}

\put(70,24){\circle{3}}
\put(72,23){$v_1$}
\put(60,12){\circle{3}}
\put(51,11){$v_4$}
\put(80,12){\circle{3}}
\put(82,11){$v_2$}
\put(70,0){\circle{3}}
\put(72,-1){$v_3$}
\put(62,35){\circle{3}}
\put(58,38){${\ell}_1$}
\put(70,35){\circle{3}}
\put(67,38){${\ell}_2$}
\put(82,35){\circle{3}}
\put(79,38){${\ell}_t$}

\put(130,24){\circle{3}}
\put(132,23){$v_1$}
\put(120,12){\circle{3}}
\put(111,11){$v_4$}
\put(140,12){\circle{3}}
\put(131,11){$v_2$}
\put(130,0){\circle{3}}
\put(132,-1){$v_3$}
\put(122,35){\circle{3}}
\put(114,34){${\ell}_1$}
\put(130,35){\circle{3}}
\put(142,35){\circle{3}}
\put(151,12){\circle{3}}
\put(151,20){\circle{3}}
\put(153.5,19){${\ell}_2$}
\put(151,0){\circle{3}}


\put(-58.5,0){\line(1,0){17}}
\put(-60,1.5){\line(0,1){12}}
\put(-40,1.5){\line(0,1){12}}
\put(-60, 16.5){\line(1,1){8.5}}
\put(-40,16.5){\line(-1,1){8.5}}
\put(-50,26.5){\line(-1,1){8}}
\put(-50,26.5){\line(0,1){8}}
\put(-50,26.5){\line(3,2){12}}

\put(1.5,0){\line(1,0){17}}
\put(0,1.5){\line(0,1){12}}
\put(20,1.5){\line(0,1){12}}
\put(0, 16.5){\line(1,1){8.5}}
\put(20,16.5){\line(-1,1){8.5}}
\put(10,26.5){\line(-1,1){8}}
\put(10,26.5){\line(0,1){8}}
\put(10,26.5){\line(3,2){12}}
\put(21.5,15){\line(1,1){8}}
\put(21.5,15){\line(1,0){8}}
\put(21.5,15){\line(2,-3){8}}

\put(70,22.5){\line(-1,-1){9}}
\put(70,22.5){\line(1,-1){9}}
\put(70,1.5){\line(1,1){9}}
\put(70,1.5){\line(-1,1){9}}
\put(70,25.5){\line(-1,1){8}}
\put(70,25.5){\line(0,1){8}}
\put(70,25.5){\line(3,2){12}}

\put(130,22.5){\line(-1,-1){9}}
\put(130,22.5){\line(1,-1){9}}
\put(130,1.5){\line(1,1){9}}
\put(130,1.5){\line(-1,1){9}}
\put(130,25.5){\line(-1,1){8}}
\put(130,25.5){\line(0,1){8}}
\put(130,25.5){\line(3,2){12}}
\put(141.5,12){\line(1,1){8}}
\put(141.5,12){\line(1,0){8}}
\put(141.5,12){\line(2,-3){8}}

\end{picture}
\end{center}
\vspace{.9in}
\caption{Unicyclic graphs $G$ with $\mathcal{C} \in \{C_5, C_4\}$ and $diam(G)=3$}\label{uniC}
\end{figure}
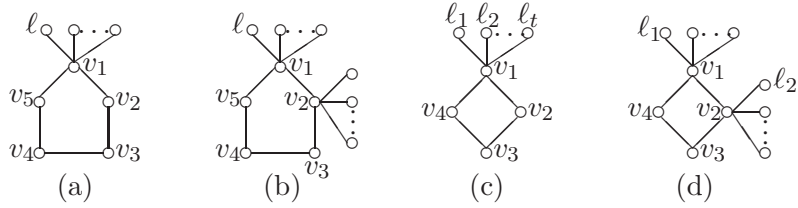

So, suppose that $m=3$; one can readily check that $G$ is isomorphic to one of the unicyclic graphs in Fig.~\ref{uniC3}. If $G$ is isomorphic to (a) of Fig.~\ref{uniC3}, i.e., $G$ is the vertex sum of $C_3$ and $S_{n-2}$ at one of the leaves of the star, we claim that $Z(G)=n-3$: (i) $Z(G) \ge n-3$ by Theorem~\ref{cutV}, since $Z(C_3)=2$ and $Z(S_{n-2})=n-4$; (ii) $Z(G) \le n-3$, since $\ell, s, v_1, v_2$ is an induced path in $G$. If $G$ is isomorphic to (b) of Fig.~\ref{uniC3}, then $\ell_1, s, v_1, \ell_t$ and $v_2, v_3$ are induced paths in $G$, and hence $Z(G) \le n-4$. If $G$ is isomorphic to (c) of Fig.~\ref{uniC3}, then, noting that $n \ge 6$, either $v_1$ or $v_2$, say $v_1$, has terminal degree at least two; then $r \ge 2$. Since $\ell_1, v_1, \ell_2$ and $v_2, v_3, \ell_t$ are induced paths in $G$, $Z(G)=P(G) \le n-4$. If $G$ is isomorphic to (d) of Fig.~\ref{uniC3}, then $\ell_1, v_1, v_2, \ell_2$ and $v_3, \ell_3$ are induced paths in $G$; thus $Z(G)=P(G) \le n-4$.~\hfill
\end{proof}

\begin{figure}[htbp]
\begin{center}
\begin{picture}(0,0)(60,55)
\setlength{\unitlength}{1.3pt}

\put(-55,-14){(a)}
\put(5,-14){(b)}
\put(65,-14){(c)}
\put(125,-14){(d)}
\put(-49,34){$\ldots$}
\put(11,34){$\ldots$}
\put(71,24){$\ldots$}
\put(131,24){$\ldots$}
\put(-5,10){$\cdot$}
\put(-5,7){$\cdot$}
\put(-5,4){$\cdot$}
\put(45,-5){$\cdot$}
\put(45,-8){$\cdot$}
\put(45,-11){$\cdot$}
\put(105,-5){$\cdot$}
\put(105,-8){$\cdot$}
\put(105,-11){$\cdot$}
\put(152,-4){$\cdot$}
\put(152,-7){$\cdot$}
\put(152,-10){$\cdot$}


\put(-60,0){\circle{3}}
\put(-69,-1){$v_3$}
\put(-40,0){\circle{3}}
\put(-38,-1){$v_2$}
\put(-50,15){\circle{3}}
\put(-47.5,14){$v_1$}
\put(-50,25){\circle{3}}
\put(-48,24){$s$}
\put(-50,36){\circle{3}}
\put(-58,36){\circle{3}}
\put(-64,35){$\ell$}
\put(-38,36){\circle{3}}

\put(0,0){\circle{3}}
\put(-3,-6){$v_3$}
\put(20,0){\circle{3}}
\put(18,-6){$v_2$}
\put(10,15){\circle{3}}
\put(12.5,14){$v_1$}
\put(10,25){\circle{3}}
\put(12,24){$s$}
\put(10,36){\circle{3}}
\put(7,39){${\ell}_2$}
\put(2,36){\circle{3}}
\put(-1,39){${\ell}_1$}
\put(22,36){\circle{3}}
\put(20,39){${\ell}_r$}
\put(-5,15){\circle{3}}
\put(-21,14){${\ell}_{r+2}$}
\put(-5,3){\circle{3}}
\put(-14,2){${\ell}_t$}
\put(-5,23){\circle{3}}
\put(-21,22){${\ell}_{r+1}$}

\put(60,0){\circle{3}}
\put(57,-6){$v_3$}
\put(80,0){\circle{3}}
\put(78,-6){$v_2$}
\put(70,15){\circle{3}}
\put(72.5,14){$v_1$}
\put(70,26){\circle{3}}
\put(67,29){${\ell}_2$}
\put(62,26){\circle{3}}
\put(59,29){${\ell}_1$}
\put(82,26){\circle{3}}
\put(80,29){${\ell}_r$}
\put(45,0){\circle{3}}
\put(29,-1){${\ell}_{r+2}$}
\put(45,-12){\circle{3}}
\put(37,-13){${\ell}_{t}$}
\put(45,8){\circle{3}}
\put(29,7){${\ell}_{r+1}$}

\put(120,0){\circle{3}}
\put(117,-6){$v_3$}
\put(140,0){\circle{3}}
\put(137,-6){$v_2$}
\put(130,15){\circle{3}}
\put(132.5,14){$v_1$}
\put(130,26){\circle{3}}
\put(122,26){\circle{3}}
\put(114,25){${\ell}_1$}
\put(142,26){\circle{3}}
\put(105,0){\circle{3}}
\put(105,-12){\circle{3}}
\put(96.5,-13){${\ell}_3$}
\put(105,8){\circle{3}}
\put(155,0){\circle{3}}
\put(155,-12){\circle{3}}
\put(155,8){\circle{3}}
\put(157.5,7){${\ell}_2$}

\put(-58.5,0){\line(1,0){17}}
\put(-51.5,15){\line(-2,-3){9}}
\put(-48.5,15){\line(2,-3){9}}
\put(-50,16.5){\line(0,1){7}}
\put(-50,26.5){\line(-1,1){8}}
\put(-50,26.5){\line(0,1){8}}
\put(-50,26.5){\line(3,2){12}}

\put(1.5,0){\line(1,0){17}}
\put(8.5,15){\line(-2,-3){9}}
\put(11.5,15){\line(2,-3){9}}
\put(10,16.5){\line(0,1){7}}
\put(10,26.5){\line(-1,1){8}}
\put(10,26.5){\line(0,1){8}}
\put(10,26.5){\line(3,2){12}}
\put(8.5,15){\line(-3,2){11.5}}
\put(8.5,15){\line(-1,0){11.5}}
\put(8.5,15){\line(-1,-1){12}}

\put(61.5,0){\line(1,0){17}}
\put(68.5,15){\line(-2,-3){9}}
\put(71.5,15){\line(2,-3){9}}
\put(70,16.5){\line(-1,1){8}}
\put(70,16.5){\line(0,1){8}}
\put(70,16.5){\line(3,2){12}}
\put(58.5,0){\line(-3,2){11.5}}
\put(58.5,0){\line(-1,0){11.5}}
\put(58.5,0){\line(-1,-1){12}}

\put(121.5,0){\line(1,0){17}}
\put(128.5,15){\line(-2,-3){9}}
\put(131.5,15){\line(2,-3){9}}
\put(130,16.5){\line(-1,1){8}}
\put(130,16.5){\line(0,1){8}}
\put(130,16.5){\line(3,2){12}}
\put(118.5,0){\line(-3,2){11.5}}
\put(118.5,0){\line(-1,0){11.5}}
\put(118.5,0){\line(-1,-1){12}}
\put(141.5,0){\line(1,-1){12}}
\put(141.5,0){\line(1,0){11.5}}
\put(141.5,0){\line(3,2){12}}

\end{picture}
\end{center}
\vspace{.9in}
\caption{Unicyclic graphs $G$ with $\mathcal{C}=C_3$ and $diam(G)=3$}\label{uniC3}
\end{figure}
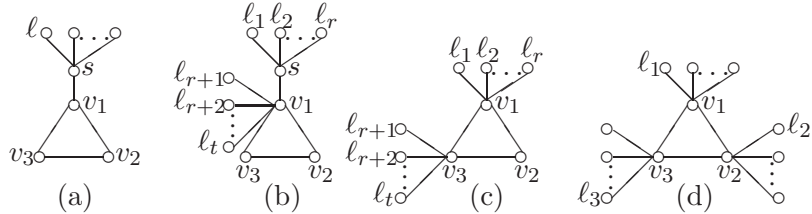

\begin{corollary}
Let $G$ be a connected, unicyclic graph of order $n\geq 5$ and having a connected $\overline{G}$. Then $Z(G)+Z(\overline{G})=2(n-3)$ if and only if $n=5$, or $n \ge 6$ and $G$ is the vertex sum of $C_3$ and $S_{n-2}$ at one of the leaves of the star.
\end{corollary}

\vskip .1in

\begin{figure}[htbp]
\begin{center}
\begin{picture}(0,0)(60,55)
\setlength{\unitlength}{1.3pt}

\put(-12,50){$G$}
\put(98,50){$\overline{G}$}
\put(12,12){$\cdot$}
\put(14,14){$\cdot$}
\put(16,16){$\cdot$}
\put(102,25){$\cdots$}
\put(112,25){$\cdot$}
\put(115,25){$\cdot$}
\put(75,21){\dashbox(64,12)}
\put(140,25){$K_{n-4}$}

\put(-20,0){\circle{3}}
\put(-28.5,-2){$\ell_2$}
\put(-20,40){\circle{3}}
\put(-29,40){$v_1$}
\put(0,0){\circle{3}}
\put(2.5,-2){$\ell_3$}
\put(0,40){\circle{3}}
\put(2,40){$v_2$}
\put(-30,20){\circle{3}}
\put(-38,19){$\ell_1$}
\put(20,20){\circle{3}}
\put(22.5,18){$\ell_{n-2}$}
\put(10,10){\circle{3}}
\put(12.5,8){$\ell_4$}

\put(100,0){\circle{3}}
\put(98,-6){$v_2$}
\put(100,40){\circle{3}}
\put(98,43){$v_1$}
\put(85,28){\circle{3}}
\put(76.5,26){$\ell_3$}
\put(100,28){\circle{3}}
\put(91.5,26){$\ell_4$}
\put(120,28){\circle{3}}
\put(122.5,26){$\ell_{n-2}$}
\put(90,10){\circle{3}}
\put(82,8){$\ell_1$}
\put(110,10){\circle{3}}
\put(112.5,8){$\ell_2$}

\put(-18.5,40){\line(1,0){17}}
\put(-30, 21.5){\line(1,2){9}}
\put(-30, 18.5){\line(1,-2){9}}
\put(-20,38.5){\line(0,-1){37}}
\put(0,38.5){\line(0,-1){37}}
\put(0,38.5){\line(1,-1){18}}
\put(0,38.5){\line(1,-3){9}}

\put(100,38.5){\line(0,-1){9}}
\put(100,38.5){\line(-3,-2){14}}
\put(100,38.5){\line(2,-1){19}}
\put(90,11.5){\line(-1,3){5}}
\put(90,11.5){\line(2,3){10}}
\put(90,11.5){\line(2,1){30}}
\put(110,11.5){\line(-3,2){24}}
\put(110,11.5){\line(-2,3){10}}
\put(110,11.5){\line(2,3){10}}
\put(100,1.5){\line(-1,1){8.5}}
\put(100,1.5){\line(1,1){8.5}}

\end{picture}
\end{center}
\vspace{.9in}
\caption{The unicyclic graph $G$ of order $n \ge 6$, with $Z(G)=Z(\overline{G})=n-3$}\label{ex2Z}
\end{figure}
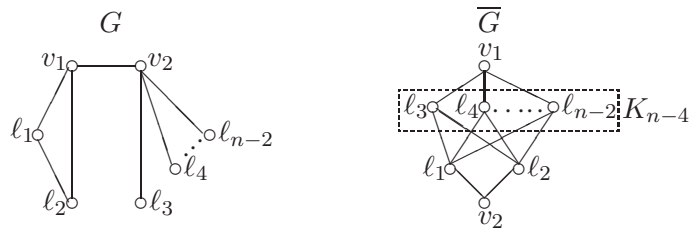

\begin{proof}
($\Longrightarrow$) It follows from Lemma~\ref{5} and Theorem~\ref{main}.

($\Longleftarrow$) If $n=5$, the result follows from Lemma \ref{5}. So, suppose that $n \ge 6$ and that $G$ is the vertex sum of $C_3$ and $S_{n-2}$ at one of the leaves of the star (see Fig.~\ref{ex2Z}). Then $Z(G)=n-3$ as shown in the proof of Theorem \ref{main}. We will show that $Z(\overline{G})=n-3$. Since both $G$ and $\overline{G}$ are connected, by Theorem \ref{connected}, $Z(\overline{G}) \le n-3$; it remains to show that $Z(\overline{G}) \ge n-3$. If we let $W_1=\{\ell_i \mid 1 \le i \le n-2 \mbox{ and } i \neq 2\}$ and $W_2= \{\ell_i \mid 2 \le i \le n-2\}$, then $\overline{G}[W_1] \cong K_{n-3} \cong \overline{G}[W_2]$, and thus, by Lemma \ref{GK}, $Z(\overline{G}) \ge n-4$. Assume that there exists a zero forcing set $S$ of $\overline{G}$ with $|S|=n-4$. Since $N_{\bar{G}}(\ell_1)=N_{\bar{G}}(\ell_2)$, $|S \cap \{\ell_1, \ell_2\}| \ge 1$. Similarly, since $N_{\bar{G}}(\ell_3)=N_{\bar{G}}(\ell_4)=\cdots=N_{\bar{G}}(\ell_{n-2})$, $|S \cap \{\ell_3, \ell_4, \ldots, \ell_{n-2}\}| \ge n-5$. Since $|S|=n-4$, without loss of generality, we may assume that $S=\{\ell_i \mid  2 \le i \le n-3\}$. But, then each vertex in $S$ has two or more white neighbors in $\overline{G}$; thus, there is no zero forcing set of cardinality $n-4$ in $\overline{G}$.~\hfill
\end{proof}

\textbf{Acknowledgements.} The authors thank two anonymous referees on an earlier draft of this paper for many constructive comments which significantly improved this paper. They also thank two additional referees for their careful reading of the paper and their helpful comments.


\begin{thebibliography}{0}

\bibitem{AIM} AIM Minimum Rank - Special Graphs Work Group (F. Barioli, W. Barrett, S. Butler, S.M. Cioab\u{a}, D. Cvetkovi\'{c}, S.M. Fallat, C. Godsil, W. Haemers, L. Hogben, R. Mikkelson, S.
Narayan, O. Pryporova, I. Sciriha, W. So, D. Stevanovi\'{c}, H. van der Holst, K. Vander Meulen, A.W. Wehe). Zero forcing sets and the minimum rank of graphs. {\it Linear Algebra Appl. } {\bf{428}} (2008) 1628-1648.

\bibitem{pathcover} F. Barioli, W. Barrett, S.M. Fallat, H.T. Hall, L. Hogben, B. Shader, P. van den Driessche and H. van der Holst, Zero forcing parameters and minimum rank problems. \emph{Linear Algebra Appl.} \textbf{433} (2010) 401-411.

\bibitem{ref2} F. Barioli, W. Barrett, S.M. Fallat, H.T. Hall, L. Hogben, B. Shader, P. van den Driessche and H. van der Holst, Parameters related to tree-width, zero forcing, and maximum nullity of a graph. \emph{J. Graph Theory} \textbf{72}, Issue 2 (2013) 146-177.

\bibitem{min-degree} A. Berman, S. Friedland, L. Hogben, U.G. Rothblum and B. Shader, An upper bound for the minimum rank of a graph. \textit{Linear Algebra Appl.} \textbf{429} (2008) 1629-1638.

\bibitem{p1} D. Burgarth and V. Giovannetti, Full control by locally induced relaxation. \emph{Phys. Rev. Lett.} \textbf{99} (2007) 100501.

\bibitem{p2} D. Burgarth and K. Maruyama, Indirect Hamiltonian identification through a small gateway. \emph{New J. Phys.} \textbf{11} (2009) 103019.

\bibitem{CEJO} G. Chartrand, L. Eroh, M.A. Johnson and O.R. Oellermann, Resolvability in graphs and the metric dimension of a graph. \textit{Discrete Appl. Math.} \textbf{105} (2000) 99-113.

\bibitem{CZ} G. Chartrand and P. Zhang, {\it Introduction to graph theory.} McGraw-Hill, Kalamazoo, MI (2004).

\bibitem{iteration} K. Chilakamarri, N. Dean, C.X. Kang and E. Yi, Iteration index of a zero forcing set in a graph. \textit{Bull. Inst. Combin. Appl.} {\bf{64}} (2012) 57-72.

\bibitem{CRG} D.G. Corneil, H. Lerchs and L.S. Burlingham, Complement reducible graphs. \textit{Discrete Appl. Math.} \textbf{3} (1981) 163-174.

\bibitem{Z+e} C.J. Edholm, L. Hogben, M. Huynh, J. LaGrange and D.D. Row, Vertex and edge spread of zero forcing number, maximum nullity, and minimum rank of a graph. \textit{Linear Algebra Appl.} \textbf{436} (2012) 4352-4372.

\bibitem{dimZ} L. Eroh, C.X. Kang and E. Yi, A comparison between the metric dimension and zero forcing number of trees and unicyclic graphs. \textit{arXiv:1408.5943}.

\bibitem{ZFsurvey} S.M. Fallat and L. Hogben, The minimum rank of symmetric matrices described by a graph: A survey. \textit{Linear Algebra Appl.} \textbf{426} (2007) 558-582.

\bibitem{ZFsurvey2} S.M. Fallet and L. Hogben, Variants on the minimum rank problem: A survey II. \textit{arXiv:1102.5142v1}.

\bibitem{ref1} S. Friedland and R. Loewy, On the minimum rank of a graph over finite fields. \emph{Linear Algebra Appl.} \textbf{436} (2012) 1710-1720.

\bibitem{proptime} L. Hogben, M. Huynh, N. Kingsley, S. Meyer, S. Walker and M. Young, Propagation time for zero forcing on a graph. \textit{Discrete Appl. Math.} \textbf{160} (2012) 1994-2005.

\bibitem{pzf} C.X. Kang and E. Yi, Probabilistic zero forcing in graphs. \emph{Bull. Inst. Combin. Appl.} \textbf{67} (2013) 9-16.

\bibitem{nordhaus} E.A. Nordhaus and J.W. Gaddum, On complementary graphs. \emph{Amer. Math. Monthly} \textbf{63} (1956) 175-177.

\bibitem{cutvertex} D.D. Row, A technique for computing the zero forcing number of a graph with a cut-vertex. \emph{Linear Algebra Appl.} \textbf{436} (2012) 4423-4432.

\bibitem{p3} S. Severini, Nondiscriminatory propagation on trees. \emph{J. Phys. A: Math. Theor.} \textbf{41} (2008) 482002.


\end{thebibliography}
\end{document}